\documentclass[english]{amsart} 

\usepackage{amsfonts,amsmath,amssymb,color,amsthm}
\usepackage[T1]{fontenc} 
\usepackage[english]{babel} 
\usepackage{color}

\usepackage[T1]{fontenc} 
\usepackage[all]{xy}

\newtheorem{theorem}[equation]{Theorem}
\newtheorem{lemma}{Lemma}[section]
\newtheorem{corollary}[lemma]{Corollary}
\newtheorem{proposition}[lemma]{Proposition}

\theoremstyle{definition}
\newtheorem{definition}[lemma]{Definition}

\theoremstyle{remark}
\newtheorem{remark}[lemma]{Remark}
\newtheorem{example}[lemma]{Example}

\newcommand\N{{\mathbb N}}
\newcommand\Cr{{\mathrm{Cr}}}
\newcommand\Bir{{\mathrm{Bir}}}
\newcommand\id{{\mathrm{id}}}

\newcommand\R{{\mathbb R}}
\newcommand\A{{\mathbb A}}

\newcommand\C{{\mathbb C}}
\renewcommand\k{{\mathrm{k}}}
\newcommand\kk{{\mathbf{k}}}
\newcommand\p{{\mathbb P}}

\newcommand\G[1]{\Bir(\p^n)_{\le #1}}

\newcommand\Proj{{\mathbb P}}
\newcommand\PGL{{\mathrm{PGL}}}

\newcommand\Aut{{\mathrm{Aut}}}

\newcommand\tr{\hbox to 1mm  {${}^t \!  $} }

\title{Topologies and structures of the Cremona groups}
\thanks{The authors gratefully acknowledge support by the Swiss National Science Foundation Grant  "Birational Geometry" PP00P2\_128422 /1 and by the French National Research Agency Grant "BirPol", ANR-11-JS01-004-01. }

\author{J\'er\'emy Blanc}
\address{J\'er\'emy Blanc, Universit\"{a}t Basel, Mathematisches Institut, Rheinsprung $21$, CH-$4051$ Basel, Switzerland.}
\email{jeremy.blanc@unibas.ch}

\author{Jean-Philippe Furter}
\address{Jean-Philippe Furter, Dpt. of Math., Univ. of La Rochelle, av. Cr\'epeau, 17000 La Rochelle, France}
\email{jpfurter@univ-lr.fr}
\subjclass[2010]{14E07, 20G15, 54H11}

\begin{document}
\maketitle

\centerline{\today}
\begin{abstract}We study the algebraic structure of the $n$-dimensional Cremona group and show that it is not an algebraic group of infinite dimension (ind-group) if $n\ge 2$. We  describe the obstruction to this, which is of a topological nature.

By contrast, we show the existence of a Euclidean topology on the Cremona group which extends that of its classical subgroups and makes it a topological group.\end{abstract}

\section{Introduction}
In \cite[\S 3]{Sha1},  I.R. Shafarevich asked:
"Can one introduce a universal structure of an infinite-dimensional group
in the group of all automorphisms (resp. all birational automorphisms)
of arbitrary algebraic variety?"\\

In an open problem session held at the international congress (see \cite{Mum}),
D.~Mumford suggested:
"Let $G= {\rm Aut}_{\C} \C (x_0,x_1)$ be the Cremona group [\ldots].
The problem is to topologize $G$ and associate to it a Lie algebra consisting,
roughly, of those meromorphic vector fields $D$ on $\Proj ^2 ( \C)$ which "integrate"
into an analytic family of Cremona transformations".\\

In $2010$,  in the question session of the workshop "Subgroups of the Cremona group" in Edinburgh,  J.-P. Serre asked the following question (see also the introduction of \cite{Fa}) "Is it possible to introduce such topology on $\Cr_2(\C)$ that is compatible with $\PGL(3,\C)$ and $\PGL(2,\C)\times \PGL(2,\C)$?"\\

Let $\k$ be a fixed field, and $n$  a positive integer.
The  $n$-dimensional Cremona group over $\k$, written $\Cr_n(\k)$,
is the group of all birational transformations of the space of dimension $n$ over $\k$ (affine or projective).
Algebraically, it corresponds to the group of $\k$-automorphisms of the field $\k(x_1,\dots,x_n)$. We are interested in the possible structures one can put on it. For instance, is $\mathrm{Cr}_n(\k)$ an algebraic group (of infinite dimension) or a topological group?\\

In the algebraic geometric setting, we would like to see $\Cr_{n}(\k)$ as an algebraic group of infinite dimension.
We will show that this is not possible when $n \geq 2$, and shall describe the obstructions to this.
For any algebraic variety $A$ defined over $\k$,
there is a natural notion of families of elements of $\Cr_n(\k)$ parametrised by $A$, or equivalently of elements of $\Cr_n(A)$, which we will recall  (Definition~\ref{Defi:Family}). These are maps $A(\k)\to \Cr_n(\k)$ compatible with the structures of algebraic varieties.

When $n=1$, $\Cr_n(\k)$ is isomorphic to the algebraic group $\PGL(2,\k)$,
and families $A\to \Cr_1(\k)$ correspond to morphisms of algebraic varieties.

For $n\ge 2$, the situation is different. Denoting by $\Cr_n(\k)_d \subseteq  \Cr_n(\k)$ the set of birational maps of degree $d$
(i.e.\ maps $f\colon \p^n_\k\dasharrow \p^n_\k$ of degree $d$),
one can verify that $\Cr_n(\k)_d$ has the structure of an algebraic variety defined over $\k$, such that families $A\to \Cr_n(\k)_d$ 
correspond to morphisms of algebraic varieties (Proposition~\ref{Prop:DegreD}). One can thus decompose $\Cr_n(\k)$ into a disjoint (infinite) union of algebraic varieties,
having unbounded dimension. However, the structure of $\Cr_n(\k)$ is more complicated; putting on $\Cr_n(\k)$ the topology induced by all families,
the group is connected \cite{Bl}.

What we could expect, is to have a structure of algebraic variety of infinite dimension (ind-algebraic variety) on $\Cr_n(\k)$,
so that families $A\to \Cr_n(\k)$ would correspond to morphisms of algebraic varieties.
We will show that this is not the case:
\begin{theorem}\label{ThmNoStructure}
For any $n\ge 2$, there is no structure of algebraic variety of infinite dimension on $\Cr_n(\k)$, such that families $A\to \Cr_n(\k)$  correspond to morphisms of algebraic varieties.
\end{theorem}
The same result holds if we replace varieties by schemes, algebraic spaces or stacks. 
In fact, the problem does not arise because of high dimension; we will see that the set $\Cr_n(\k)_{\le d}$ of maps of degree $\le d$ is not an algebraic variety (neither an algebraic space nor a stack), when $d,n\ge 2$ (Proposition~\ref{Prop:led}), even if the $\Cr_n(\k)_i$ are algebraic varieties for $i=1,\dots,d$. The bad structure comes from the degeneration of maps of degree $d$ into maps of smaller degree.\\

We can even show that the obstruction  comes only from the topology:

\begin{theorem}\label{thmZar}
If $\kk$ is algebraically closed, there is no $\kk$-algebraic variety of infinite dimension which is homeomorphic to $\Cr_n(\kk)$.
\end{theorem}

The problem is, roughly speaking, that any point $\varphi\in\Cr_n(\kk)$ is contained in a closed subset $F$ of dimension $2$,
such that all curves of $F$ pass through $\varphi$.
See Section~\ref{Sec:TopObs} for more details on the topological obstruction.\\

The question of J.-P. Serre,
is whether $\Cr_2(\C)$ is a topological group, with a "natural topology", i.e.\ with a topology which induces on $\Aut(\p^2)=\PGL(3,\C)$
and $\Aut(\p^1\times \p^1)^0=\PGL(2,\C)\times \PGL(2,\C)$ the Euclidean topology.\footnote{Of course, here it cannot be the Zariski topology, since $\PGL(3,\C)$ and $\PGL(2,\C)$ are not topological groups when endowed with it.} We will give a positive answer to this question, after constructing in Section~\ref{Sec:Trans} a "Euclidean topology" on $\Cr_n(\k)$, for any local field~$\k$:

\begin{theorem}
For any $n\ge 1$, and any $($locally compact$)$ local field $\k$, there is a natural topology on $\Cr_n(\k)$, called \emph{the Euclidean topology}, which makes it a Hausdorff topological group,
and whose restriction to algebraic subgroups $($in particular to $\Aut(\p^n)=\PGL(n+1,\k)$ and $\Aut((\p^1)^n)^0=\PGL(2,\k)^n)$ is the classical Euclidean topology.
\end{theorem}

The authors thank Michel Brion, Hanspeter Kraft and Immanuel Stampfli for interesting discussions during the preparation of this article.

\section{The algebraic structure and Zariski topology of $\Cr_n(\kk)$}
\label{ZarTop}
\subsection{Definition of the topology via families and morphisms}
In this section, all our varieties will be algebraic varieties defined over some fixed algebraically closed  field $\kk$. In fact, all constructions made here are field-independent, and the generalisation to arbitrary fields can be carried out in the usual way. The topology on the varieties will only be the Zariski topology.

We recall the notion of families of birational maps, introduced by M.~Demazure in \cite{De} (see also \cite{Se}, \cite{Bl}).

\begin{definition} \label{Defi:Family}
Let $A,X$ be irreducible algebraic varieties, and let $f$ be a $A$-birational map of the $A$-variety $A\times X$, inducing an isomorphism $U\to V$, where $U,V$ are open subsets of $A\times X$, whose projections on $A$ are surjective.

The rational map $f$ is given by $(a,x)\dasharrow (a,p_2(f(a,x)))$, where $p_2$ is the second projection, and for each $\k$-point $a\in A$, the birational map $x\dasharrow p_2(f(a,x))$ corresponds to an element  $f_a\in \Bir(X)$.
The map $a\mapsto f_a$ represents a map from $A$ $($more precisely from the $A(k)$-points of $A)$ to $\Bir(X)$, and will be called a \emph{morphism} from $A$ to $\Bir(X)$.
\end{definition}
These notions yield the natural Zariski topology on $\Bir(X)$, introduced by M.~Demazure \cite{De} and J.-P. Serre \cite{Se}:
\begin{definition}  \label{defi: Zariski topology}
A subset $F\subseteq \Bir(X)$ is closed in the Zariski topology
if for any algebraic variety $A$ and any morphism $A\to \Bir(X)$ the preimage of $F$ is closed.
\end{definition}

Moreover, any birational map $X\dasharrow Y$ yields a homeomorphism between $\Bir(X)$ and $\Bir(Y)$, and we can make the following observations:
\begin{enumerate}
\item
For any $\varphi\in \Bir(X)$ the maps $\Bir(X)\to \Bir(X)$ given by $\psi\mapsto \psi \circ \varphi$, $\psi\mapsto \varphi\circ \psi$ and $\psi\mapsto \psi^{-1}$ are homeomorphisms. 
\item
It is equivalent to work with $\Bir(\A^n_\kk)$ or $\Bir(\p^n_\kk)$; we obtain thus the Zariski topology on the Cremona group  $\Cr_n(\kk)=\Bir(\A^n_\kk)$.
\end{enumerate}

The remainder of this section consists of a description of the topology of $\Cr_n(\kk)$ and its algebraic structure. 
\subsection{Bounded degree subsets}
Recall that a birational transformation $f$ of $\p^n=\p^n_\kk$ is given by 
$$h\colon(x_0:\dots:x_n)\dasharrow (h_0(x_0,\dots,x_n):\dots: h_n(x_0,\dots,x_n)),$$ where the $h_i$ are homogeneous polynomials of the same degree. Choosing the $h_i$ without common component, the degree of $h$ is the degree of the $h_i$.

\bigskip

We denote by $\Bir(\p^n)_{\le d}$ (respectively $\Bir(\p^n)_d$) the set of elements of $\Bir(\p^n)$ of degree $\le d$ (respectively of degree $d$), and have an increasing sequence
$$\Aut(\p^n)=\Bir(\p^n)_{\le 1}\subseteq \Bir(\p^n)_{\le 2}\subseteq \Bir(\p^n)_{\le 3} \subseteq \dots$$
whose union gives the group $\Bir(\p^n)$. As we will see, each $ \Bir(\p^n)_{\le d}$ is closed in $\Bir(\p^n)$ and the topology of $\Bir(\p^n)$ is the inductive topology induced by the above sequence. It then suffices to describe the topology of  $\Bir(\p^n)_{\le d}$ in order to understand the topology of $\Bir(\p^n)$.

\bigskip

\begin{definition}\label{DefWHG}
Let $d$ be a positive integer.
\begin{enumerate}
\item
We define $W_d$ to be the set of equivalence classes of non-zero  $(n+1)$-uples $(h_0,\dots,h_n)$
of homogeneous polynomials $h_i\in \kk[x_0,\dots,x_n]$ of degree $d$,
where $(h_0,\dots,h_n)$ is equivalent to $(\lambda h_0,\dots,\lambda h_n)$ for any $\lambda\in \kk^{*}$.
The equivalence class of $(h_0,\dots,h_n)$ will be denoted by $(h_0:\dots:h_n)$.
\item
We define $H_d\subseteq W_d$ to be the set of elements $h=(h_0:\dots:h_n)\in W_d$
such that the rational map
$\psi_h\colon \p^n\dasharrow \p^n$ given by $(x_0:\dots:x_n)\dasharrow
(h_0(x_0,\dots,x_n):\dots:h_n(x_0,\dots,x_n))$ is birational.
We denote by $\pi_d$ the map $H_d\to \Bir(\p^n_\kk)$ which sends $h$ onto $\psi_h$.
\end{enumerate}
\end{definition}

\begin{lemma} \label{lem:WHalgebraic}
Let $W_d,H_d$ be as in Definition~$\ref{DefWHG}$.
Then, the following holds:

\begin{enumerate}
\item
The set $W_d$ is  isomorphic to $\p^{r}$,
where $r=(n+1) \cdot \left(\begin{array}{c}d+n\\ d \end{array}\right)-1.$
\item
The set $H_d$ is locally closed in $W_d$, and thus inherits from $W_d$ the structure of an algebraic variety. 
\item
The map $\pi_d\colon H_d\to \Bir(\p^n)$ is a morphism $($in the sense of Definition~$\ref{Defi:Family})$. Its image is the set $\Bir(\p^n)_{\le d}$ of birational transformations of degree $\le d$.
\item
For any $\varphi \in \Bir(\p^n)_{\le d}$, the set $(\pi_d)^{-1}(\varphi)$ is closed in $W_d$ $($hence in $H_d)$.
\item
If $F\subseteq H_m$ $(m\ge 1)$ is closed, then $(\pi_d)^{-1}(\pi_m(F))$ is closed in $H_d$.
\end{enumerate}
\end{lemma}

\begin{proof}
Assertion $(1)$ follows from the fact that the set of homogeneous polynomials of degree $d$ in $n+1$ variables is equal to a $\kk$-vector space of dimension
${\left(\begin{array}{c}d+n \\ d \end{array}\right)}$.

$(2)$ We denote by $Y\subseteq  W_{d^{n-1}} \times W_d$ the set consisting of elements $(g,f)$, such that $h:=(g_0(f_0,\ldots,f_n) , \ldots , g_n(f_0,\ldots,f_n) )$ is a multiple (maybe $0$) of the identity, i.e. $h_ix_j=h_jx_i$ for all $i,j$. This set of equalities yields the existence of a homogeneous polynomial $a$ of degree $d^n-1$ such that $h_i=ax_i$ for each $i$. If $a$ is non-zero, then $\psi_f$ and $\psi_g$ are birational, and inverses of each other. If $a$ is zero, then $\psi_f$ contracts the entire set $\mathbb{P}^n$ onto a strict subvariety, which is included in the set where $g_i=0$ for each $i$. In particular, for any $(g,f)\in Y$, the element $\psi_f$ is birational if and only if its Jacobian is not zero (the Jacobian is the determinant of the matrix obtained from the partial derivatives of the components of $f$).

Since the inverse of every birational map of $\p^n$ of degree $d$ has degree $\le d^{n-1}$ (\cite[Theorem 1.5, page 292]{BCW}), any element $f\in H_d$ corresponds to (at least) one pair $(g,f)$ in $Y$.

The description of $Y$ shows that it is closed in $ W_{d^{n-1}} \times W_d $. Since $W_{d^{n-1}}$ is a complete variety, the projection $p_2 \colon W_{d^{n-1}} \times W_d  \to W_d$ is a Zariski-closed morphism, so $p_2(Y)$ is closed in $W_d$.  Denote by $U\subseteq W_d$ the open set of elements having a non-zero Jacobian.  By construction, we have $H_d=U\cap p_2(Y)$. This implies that $H_d$ is locally closed in $W_d$ (and that it is closed in $U$).

$(3)$ Let $f\colon H_d \times \p^n\dasharrow H_d\times \p^n$ be the $H_d$-rational map given by
$$(h,x) \dasharrow (h, h(x)).$$ 
Denote by $J$ the polynomial which is the determinant of the matrix $(\partial h_i/\partial x_j)_{i,j=0}^n$, and let $V\subseteq H_d\times \p^n$ be the open set where $J$ is not zero. We claim that the restriction of $f$ to $V$ is an open immersion, which will show that $\pi_d\colon H_d\to \Bir(\p^n)$ is a morphism. To prove this claim, we need to see that $f$ is radicial (i.e.\ universally injective) and \'etale \cite[Th\'eor\`eme 17.9.1, page 79]{Gro}. For any extension $K$ of $\kk$, and any point $h\in H_d(K)$, there exists $h'\in H_{m}(K)$ without common factors such that the composition $h'\circ h$ is $(x_0R:\dots:x_nR)$ for some non-zero polynomial $R$. The hypersurface $R=0$ is contracted by $h$ onto the set of base-points of $h'$, with codimension $\ge 2$, so $R$ can be expressed as a product of some factors of the Jacobian of $h$. This shows that the extension of $f$ to $K$ restricts to an open immersion on $V(K)\cap (\{h\} \times \p^n)$. In particular, $f$ is radicial.
The fact that the projection $H_d\times \p^n\to H_d$ is smooth and that the derivative with respect to $x$ is injective at any point of $V$ implies that $f$ is \'etale on $V$ \cite[Corollaire 17.11.2, page 84]{Gro}. This completes the proof of the claim. It then follows from the construction of $H_d$ that the image of $\pi_d$ is the set of birational transformations of degree $\le d$.

$(4)$ Let $\varphi$ be an element of $\Bir(\p^n)_{\le d}$. It corresponds to a birational map 
$\psi_h\colon \p^n\dasharrow \p^n$ given by $(x_0:\dots:x_n)\dasharrow
(h_0(x_0,\dots,x_n):\dots:h_n(x_0,\dots,x_n))$, for some homogeneous polynomials of degree $k\le d$, having no common divisor.

We observe that $(\pi_d)^{-1}(\varphi)\subseteq H_d$ is the set of elements $(g_0:\dots:g_n)$ of $W_d$ satisfying  $g_i h_j=h_ig_j$ for all $i,j$. This set is therefore closed in $W_d$, and hence in $H_d$.

$(5)$ For any positive integer $m$ and any closed subset $F \subseteq H_m$, we denote by $Y_F$ the subset of $Y\times \overline{F}$ (where $Y\subseteq W_{d^{n-1}}\times  W_d$ is as above and $\overline{F}$ is the closure of $F$ in $W_m$) consisting of elements $((g,f),h)$ where $f=(f_0:\dots:f_n)$ and $h=(h_0:\dots:h_n)$ yield the same map  $\p^n\dasharrow \p^n$. This corresponds to saying that $f_ih_j=f_jh_i$ for all $i,j$, hence $Y_F$ is  closed in  $Y\times \overline{F}$, and also in $W_{d^{n-1}}  \times  W_d\times W_m$.
Let $p_2 :W_{d^{n-1}}  \times  W_d \times W_m \to  W_d $ be the second projection. The subset $p_2 (Y_F)$ of  $W_d$ is closed in
$W_d$ and also in $p_2(Y)$. Intersecting with $U$, we see that $p_2(Y_F)\cap U$ is closed in $p_2(Y)\cap U=H_d$. By construction, $p_2(Y_F)\cap U= (\pi_d)^{-1}(\pi_m(F))$, which is thus closed in $H_d$.
\end{proof}

\begin{remark}
Lemma~\ref{lem:WHalgebraic} shows that $W_d,H_d$ are naturally algebraic varieties.  We will show that the same holds for the set $\Bir(\p^n)_{d}$ of birational transformations of degree $d$, but not for $\Bir(\p^n)_{\le d}$ which cannot be viewed as an algebraic variety for $d,n\ge 2$. However, the topology of $ \Bir(\p^n)_{\le d}$ is given by the map $\pi_d \colon H_d \to \Bir(\p^n)_{\le d}$, in the sense that $\pi_d$ is a topological quotient map.  \end{remark}

\subsection{\boldmath Liftings of morphisms from $\Bir(\p^n)$ to the $H_d$ and description of the Zariski topology}
The following technical lemma will be used to deduce three corollaries that describe the topology of $\Bir(\p^n)$.

\begin{lemma}\label{Lem:ReleveLocalBirPn}
Let $A$ be an irreducible algebraic variety and $\rho\colon A\to \Bir(\p^n)$ be a morphism. There exists an open affine covering $(A_i)_{i \in I}$ of $A$ such that for each $i$, there exists an integer $d_i$ and a morphism $\rho_i \colon A_i\to H_{d_i}$  such that the restriction of $\rho$ to $A_i$ is equal to $\pi_{d_i}\circ \rho_i$. \end{lemma}
\begin{proof}
Let $\tau\colon A\to \Bir(\p^n)$ be a morphism given by a $A$-birational map $f\colon A\times \p^n\dasharrow A\times \p^n$, which restricts to an open immersion on an open set $U$. 
Let $a_0\in A$ be some given point and let $A_0\subseteq A$ be an open affine set containing $a_0$. We also fix an element $w_0=(a_0,y)\in U$, and fix homogeneous coordinates $(x_0:\dots:x_n)$ on $\p^n$ such that $y=(1:0:\dots:0)$ and that $f(w_0)$ does not belong to the plane $x_0=0$. We then denote by $\A^n\subseteq \p^n$ the affine set where $x_0=1$, which has natural affine coordinates $z_1=\frac{x_1}{x_0},\dots,z_n=\frac{x_n}{x_0}$. 
In these coordinates, $f$ restricts to a rational map $A_0\times \A^n\dasharrow \A^n$ which is defined at $w_0$. Its composition with the projection on the $i$-th coordinate is a rational function on $A_0\times \A^n$, which is defined at $w_0$. We obtain  that $f|_{A_0\times \A^n}$ can be written, in a neighbourhood of $w_0$, as 
$(a,(z_1,\dots,z_n))\mapsto \left(\frac{P_1}{Q_1},\dots,\frac{P_n}{Q_n}\right)$
for some $P_i,Q_i\in \kk[A_0][z_1,\dots,z_n]$ such that none of the $Q_i$ vanish at $w_0$. Homogenising the description, we see that $f$ is given, in a neighbourhood of $w_0$, by
$$(a,(x_0:\dots : x_n))\mapsto \left(h_0:\dots:h_n\right)$$
where the $h_i\in \kk[A_0][x_0,\dots,x_n]$ are homogeneous polynomials in $x_0,\dots,x_n$ of the same degree $d_0$, such that not all vanish at $w_0$.
Let $U_0$  be the set  of points of  $(A_0\times \p^n)\cap U$ where at least one of the $h_i$ does not vanish. It is an open subset of $A\times \p^n$.
Its projection $p_1(U_0)$ on $A$ is an open subset of $A_0$ containing $a_0$.
Therefore, there exists an affine open subset $\tilde{A}_0\subseteq p_1(U_0)$ containing $a_0$.
The $n$-uple $(h_0,\dots,h_n)$ gives rise to a morphism $\rho_0\colon\tilde{A}_0\to H_d$. By construction, the restriction of $\rho$ to $\tilde{A}_0$ is equal to $\pi_{d}\circ \rho_0$. 
Repeating the process for each point of $A$, we obtain an open affine covering.
\end{proof}

\begin{corollary}\label{Cor:ClosedPn}
A set $F\subseteq \Bir(\p^n)$ is closed if and only if $(\pi_d)^{-1}(F)$ is closed in $H_d$ for any $d\ge 1$. 
\end{corollary}
\begin{proof}
By definition, $F$ is closed in $\Bir(\p^n)$ if and only if its preimage by any morphism is closed. Because each $\pi_d$ is a morphism, we see that $(\pi_d)^{-1}(F)$ is closed in $H_d$ if $F$ is closed in $\Bir(\p^n)$.

Suppose now that $(\pi_d)^{-1}(F)$ is closed in $H_d$ for each $d$, and let us prove that $F\subseteq \Bir(\p^n)$ is closed. We let $\rho\colon A\to \Bir(\p^n)$ be a morphism. We apply
Lemma~\ref{Lem:ReleveLocalBirPn} and obtain an open affine covering $(A_i)_{i \in I}$ of $A$ such that for each $i$, there exist an integer $d_i$ and a morphism $\rho_i \colon A_i\to H_{d_i}$  such that the restriction of $\rho$ to $A_i$ is equal to $\pi_{d_i}\circ \rho_i$. Since $\pi_{d_i}^{-1}(F)$ is closed and $\rho^{-1}(F)\cap A_i$ is equal to $(\rho_i)^{-1}(\pi_{d_i}^{-1}(F))$, we obtain that $\rho^{-1}(F)\cap A_i$ is closed in $A_i$ for each $i$. Consequently, $\rho^{-1}(F)$ is closed in $A$, as required.
\end{proof}
\begin{corollary}\label{Cor:Degreleclosed}
For any $d$, the set $\Bir(\p^n)_{\le d}$ of birational transformations of degree $\le d$ is closed in $\Bir(\p^n)$.
\end{corollary}
\begin{proof}
By Corollary~\ref{Cor:ClosedPn}, it suffices to see that $(\pi_m)^{-1}(\Bir(\p^n)_{\le d})=(\pi_m)^{-1}(\pi_d(H_d))$ is closed in $H_m$ for any $m$. This follows from Lemma~\ref{lem:WHalgebraic}. 
\end{proof}

\begin{corollary}\label{Coro:QuotientBirPnd}
For any $d$, the map $\pi_d\colon H_d\to \Bir(\p^n)_{\le d}$ is surjective, continuous and closed.
In particular, it is a topological quotient map.
\end{corollary}

\begin{proof}
The surjectivity follows from the construction of $H_d$ and $\pi_d$ (see Lemma~\ref{lem:WHalgebraic}). The continuity of $\pi_d$ follows from the fact that it is a morphism in the sense of Definition~$\ref{Defi:Family}$ (Lemma~\ref{lem:WHalgebraic}). Let $F\subseteq H_d$  be closed. By Lemma~\ref{lem:WHalgebraic}, $(\pi_m)^{-1}(\pi_d(F))$ is closed in $H_m$ for any $m$, so $\pi_d(F)$ is closed in $\Bir(\p^n)$ (Corollary~\ref{Cor:ClosedPn}). 
\end{proof}

\begin{proposition}\label{Prop:MainTop}
The topology of $\Bir(\p^n)$ is the inductive limit topology given by the Zariski topologies of $\Bir(\p^n)_{\le d}$, $d\in \mathbb{N}$, which are the quotient topology of $\pi_d\colon H_d\to\Bir(\p^n)_{\le d}$, where $H_d$ is endowed with its Zariski topology.
\end{proposition}

\begin{proof}
This follows from Corollaries~\ref{Cor:ClosedPn} and \ref{Coro:QuotientBirPnd}.
\end{proof}

\subsection{More about the liftings}
Let us show that the liftings to $H_d$ do not behave as simply as they might seem to.

\begin{example}\label{Example:Nodal}
Let $C\subseteq \p^2$ be the nodal cubic given by $abc=a^3+b^3$, where $(a:b:c)$ are homogeneous coordinates on $\p^2$ and let $\rho_3\colon C\to H_3$ be given by
\begin{center}$(a:b:c)\mapsto (x_0R:x_1S:x_2R:\dots:x_nR),$
$\begin{array}{ccc}
R=ax_2^2+cx_0x_2+bx_0^2,&&
S=ax_2^2+(b+c)x_0x_2+(a+b)x_0^2.\end{array}$\end{center}
\end{example}
\begin{lemma}
The morphism $\rho_3$ of Example~$\ref{Example:Nodal}$ has the following properties:
\begin{enumerate}
\item
The morphism $\rho=\pi_3\circ \rho_3\colon C\to \Bir(\p^n)$ satisfies $\rho(C)\subseteq \Bir(\p^n)_{\le 2}$.
\item
There is no morphism $\rho_2 \colon C\to H_2$ such that $\pi_2\circ \rho_2=\rho$.
\end{enumerate}
\end{lemma}
\begin{proof}
The image of $(0:0:1)\in C$ under $\rho_3$ is equal to $(x_0(x_0x_2):x_1(x_0x_2):\dots:x_n(x_0x_2))\in W_3$, which corresponds to the identity.

On the open subset $C\backslash \{(0:0:1)\}$ of $C$, neither $a$ nor $b$ is zero, so the restriction of $\rho_3$ to $C\backslash \{(0:0:1)\}$ corresponds to 
\begin{center}
$(a:b:c)\mapsto (x_0R':x_1S':x_2R':\dots:x_nR'),$
$\begin{array}{llllll}
R'&=&abR&=&a^2bx_2^2+abcx_0x_2+ab^2x_0^2,\\
S'&=&abS&=&a^2bx_2^2+(ab^2+abc)x_0x_2+ab(a+b)x_0^2.\end{array}$\end{center}
We now observe  that $R'=a^2bx_2^2+(a^3+b^3)x_0x_2+ab^2x_0^2=(a^2x_2+b^2x_0)(bx_2+ax_0)$ and $S'=a^2bx_2^2+(a^3+b^3+ab^2)x_0x_2+ab(a+b)x_0^2=(a^2x_2+b(a+b)x_0)(bx_2+ax_0)$.

This shows that $\pi_3(\rho_3(a:b:c))$ has degree $2$ for any $(a:b:c)\in C\backslash \{(0:0:1)\}$.

We have thus proved $(1)$. To prove $(2)$, assume the existence of a morphism $\rho_2\colon C\to H_2$ such that $\pi_2\circ \rho_2=\rho$. Since $\pi_2$ is bijective on $(\pi_2)^{-1}(\Bir(\p^n)_2)$ and  $\rho(C\backslash \{(0:0:1)\})\subseteq \Bir(\p^n)_2$, the restriction of $\rho_2$ to $C\backslash \{(0:0:1)\}$ is given by 
\begin{center}$(a:b:c)\mapsto (x_0R'':x_1S'':x_2R'':\dots:x_nR''),$
$\begin{array}{ccc}
R''=a^2x_2+b^2x_0,&&
S''=a^2x_2+b(a+b)x_0.\end{array}$\end{center}
It remains to see that this morphism does not extend to $C$. Let $\varphi\colon \p^1 \to C$ be the birational morphism which sends $(u:v)$ onto $(u^2v:uv^2:u^3+v^3)$. The hypothetic morphism $\rho_2\circ\varphi\colon \p^1 \to H_2$ sends $(u:v)$ onto  
$$(x_0(u^2x_2+v^2x_0):x_1(u^2x_2+v(u+v)x_0):x_2(u^2x_2+v^2x_0):\dots:x_n(u^2x_2+v^2x_0)).$$
In particular, the points $\rho_2\varphi((1:0))$ and $\rho_2\varphi((0:1))$ are distinct, which is impossible, since $\varphi((1:0))=\varphi((0:1))$.
\end{proof}

\subsection{Universal property with fixed degree}
Example~\ref{Example:Nodal} provides a morphism from an algebraic variety to  $\Bir(\p^n)_{\le 2}$ which does not lift to $H_2$. This phenomenon occurs because of the degeneration of elements of $\Bir(\p^2)_2$ to elements of $\Bir(\p^2)_1$. The situation is better if the degree is fixed. 
\begin{lemma}\label{Lem:Hd1d2}
Let $d_1,d_2$ be  integers with $1\le d_1\le d_2$. The set $$H_{d_1,d_2}=(\pi_{d_2})^{-1}(\Bir(\p^n)_{d_1})$$ is locally closed in $H_{d_2}$, and isomorphic to $H_{d_1,d_1}\times \p(\kk[x_0,\dots,x_n]_{d_2-d_1})$, where $\kk[x_0,\dots,x_n]_{d_2-d_1}$ denotes the $\kk$-vector space of homogeneous polynomials of degree $d_2-d_1$ in $x_0,\dots,x_n$. Moreover, the projection  
$$\rho_{d_1,d_2}\colon H_{d_1,d_2}\to H_{d_1,d_1}\subseteq H_{d_1}$$
is such that  $\pi_{d_1}\circ \rho_{d_1,d_2}$ and $\pi_{d_2}$ coincide on $H_{d_1,d_2}$.
\end{lemma}

\begin{proof}
By Corollary~\ref{Cor:Degreleclosed}, $(\pi_{d_2})^{-1}(\Bir(\p^n)_{\le d_1})$ and $(\pi_{d_2})^{-1}(\Bir(\p^n)_{\le d_1-1})$ are closed in $H_{d_2}$. This implies that $H_{d_1,d_2}$ is locally closed in $H_{d_2}$ and thus inherits  from $H_{d_2}$ a structure of algebraic variety. We have a natural bijective morphism
$$\begin{array}{ccl}
\tau\colon H_{d_1,d_1}\times \p(\kk[x_0,\dots,x_n]_{d_2-d_1})&\to& H_{d_1,d_2}\\
((f_0:\dots:f_n), h)&\to& (f_0h:\dots:f_nh),\end{array}$$
which is the restriction of a morphism

$$\begin{array}{ccl}
\hat{\tau}\colon  W_{d_1}  \times \p(\kk[x_0,\dots,x_n]_{d_2-d_1})& \to & W_{d_2}\\
((f_0:\dots:f_n),  h) & \to & (f_0 h : \dots: f_nh).\end{array}$$
Denoting by $U_{d_1}\subseteq W_{d_1}$  the open subset given by the elements $(f_0:\dots:f_n)$ such that the $f_i$ have no common factor, it suffices to see that $\hat{\tau}$ restricts to an isomorphism of $V=U_{d_1} \times \p(\kk[x_0,\dots,x_n]_{d_2-d_1})$ with its image in order to prove the result.

The morphism $\hat{\tau}$ is proper, since it is projective. Since $V=\hat{\tau}^{-1}(\hat{\tau}(V))$, the restrictions of $\hat{\tau}$ gives a proper morphism $V\to \hat{\tau}(V)$. It can be checked with differentials that it is unramified; it is moreover clearly universally bijective. Hence, it is an isomorphism.
\end{proof}

\begin{corollary}\label{Lem:ReleveGlobalDegreD}
Let $A$ be an irreducible algebraic variety and $\rho\colon A\to \Bir(\p^n)$ be a morphism, whose image is contained in $\Bir(\p^n)_d$ for some $d$. There exists a unique morphism $\tilde{\rho} \colon A\to H_{d}$  such that  $\rho=\pi_{d}\circ \tilde{\rho}$.
\end{corollary}

\begin{proof}
Lemma~\ref{Lem:ReleveLocalBirPn} yields an open affine covering $(A_i)_{i \in I}$ of $A$ with the property that for each $i$, there exist an integer $d_i$ and a morphism $\rho_i \colon A_i\to H_{d_i}$  such that the restriction of $\rho$ to $A_i$ is equal to $\pi_{d_i}\circ \rho_i$. Since $\rho(A)$ is contained in $\Bir(\p^n)_d$, we have $d_i\ge d$ for each $i$ and the image of $\rho_i$ is contained in $H_{d,d_i}=(\pi_{d_i})^{-1}(\Bir(\p^n)_{d})$. We can thus replace $\rho_i$ by $\rho_{d,d_i}\circ \rho_i$, where $\rho_{d,d_i}\colon H_{d,d_i}\to H_{d,d}$ is given in Lemma~\ref{Lem:Hd1d2}. 

After this replacement, each $d_i$ is equal to $d$. Because $\pi_d$ restricts to a bijection $H_{d,d}\to \Bir(\p^n)_{d}$, the maps $\rho_i$ and $\rho_j$ coincide on $A_i\cap A_j$. In particular, they yield a morphism $\tilde{\rho} \colon A\to H_{d}$  such that  $\rho=\pi_{d}\circ \tilde{\rho}$. Its uniqueness also follows from the fact that $\pi_d$ restricts to a bijection $H_{d,d}\to \Bir(\p^n)_{d}$.
\end{proof}

This yields the following result, also proved in \cite[$\S 3.4$]{Ngu} (by other methods).

\begin{proposition}  \label{Prop:DegreD}
Let $d\ge 1$ be an integer. The following hold:
\begin{enumerate}
\item
The map $\pi_d\colon H_d\to \Bir(\p^n)_{\le d}$ restricts to a bijection $$H_{d,d}=(\pi_d)^{-1}(\Bir(\p^n)_{d})\to \Bir(\p^n)_{d},$$
where $H_{d,d}$ is open in $H_d$, and is thus an algebraic variety.
\item
Let $A$ be an irreducible algebraic variety.
The morphisms $A\to \Bir(\p^n)$ whose image is in $\Bir(\p^n)_{d}$
correspond, via $\pi_d$, to the morphisms of algebraic varieties $A\to H_{d,d}$. $($This again shows that $\pi_d$ restricts to a homeomorphism $H_{d,d}\to \Bir(\p_n)_d.)$\end{enumerate}
\end{proposition}

\begin{proof}
$(1)$ Since $(\pi_d)^{-1}(\Bir(\p^n)_{\le d-1})$ is closed in $H_d$, $(\pi_d)^{-1}(\Bir(\p^n)_{d})$ is open in $H_d$ and is thus an algebraic variety.

$(2)$ If  $\rho \colon A\to H_{d,d}$ is a morphism of algebraic varieties, then $\pi_d\circ \rho$ is a morphism $A\to\Bir(\p^n)$ in the sense of Definition~\ref{Defi:Family}, having its image in $\Bir(\p^n)_d$.

Conversely, let $\rho \colon A\to \Bir(\p^n)$ be a morphism having its image in $\Bir(\p^n)_d$. It corresponds to a unique map $\tilde{\rho}\colon A\to H_{d,d}$ such that $\rho=\pi_d \circ \tilde{\rho}$. Corollary~\ref{Lem:ReleveGlobalDegreD} implies that $\tilde{\rho}$ is a morphism of algebraic varieties. 
\end{proof}

\begin{remark}
One can see that $\Bir(\p^2_\mathbb{C})_{d}$ is irreducible if and only if $d\le 3$ and is connected when $d\le 6$ (\cite{CD}, \cite{Ngu}, \cite{BCM}); the connectedness is open for large $d\ge 7$.
The variety $\Bir(\p^3_\mathbb{C})_{2}$ has three irreducible components \cite[Proposition 2.4.1]{PRV},
but is connected. The structure of $\Bir(\p^n_\mathbb{C})_{d}$  for $d,n\ge 3$ is far from being understood.

Note that the identity is in the closure of $\Bir(\p^n)_d$ for any $d$ \cite[Lemma~3.3]{Bl}, so  $\overline{\Bir(\p^n)_2}=\Bir(\p^n)_{\le 2}$. For large $d$, $\overline{\Bir(\p^n)_d}\subsetneq \Bir(\p^n)_{\le d}$ seems more plausible.
\end{remark}

\subsection{Algebraic subgroups}\label{SubSec:algsubgroups}

\begin{proposition}  \label{Prop:ClosedSubgroups}
Let $G\subseteq \Bir(\p^n)$ be a subgroup, which is closed and connected for the Zariski topology, and with $G\subseteq \Bir(\p^n)_{\le d}$ for some finite degree $d$.

Choosing $d$ minimal, the set $(\pi_d)^{-1}(G\cap \Bir(\p^n)_{d})$ is non empty.
Let $K$ denote its closure in $H_d$. It has the following properties:

\begin{enumerate}
\item
$\pi_d$ induces a homeomorphism $K\to G$ $($for the Zariski topology$)$;
\item
Let $A$ be an irreducible algebraic variety.
The morphisms $A\to \Bir(\p^n)$  $($in the sense of Definition~$\ref{Defi:Family})$ having image in $G$
correspond, via $\pi_d$, to the morphisms of algebraic varieties $A\to K$;

\item The liftings to $K$ of the product map $G \times G \to G$, $(f,g) \mapsto f \circ g$ and of the inverse map $G \to G$, $f \mapsto f^{-1}$ give rise to morphisms of algebraic varieties $K\times K\to K$ and  $ K\to K$.
\end{enumerate}
This gives  $G$ a unique structure of algebraic group.
\end{proposition}

\begin{proof}
Since $G$ is closed in $\Bir(\p^n)_{\le d}$, its pull-back $(\pi_d)^{-1}(G)$ is closed in $H_d$ and  thus has a finite number of irreducible components $I_1,\dots,I_r$. By Corollary~\ref{Coro:QuotientBirPnd}, the sets $\pi_d(I_1),\dots,\pi_d(I_r)$ are closed and irreducible and they cover $G$. Keeping only the maximal ones, we obtain the irreducible components of $G$, which we now call $Z_1,\ldots,Z_m$.

For $i=1,\dots,m$, we set $A_i=\{ g \in G\ |\  g Z_1 = Z_i\}$, which is equal to $\{g\in G\ |\ gZ_1 \subseteq Z_i \}=\bigcap_{h \in Z_1} Z_i h^{-1}$ and is thus closed in $G$. Since $G$ is the disjoint union of all $A_i$, each $A_i$ is also open. There is thus an $i$ with $A_i=G$, which implies that $GZ_1=Z_i$. Since $GZ_1=G$, we get $m=1$ and see that $G$ is irreducible.

Since $G$ is irreducible, the open subset $G_d=\Bir(\p^n)_d\cap G$  is irreducible, and  dense in $G$. Writing $K_d=(\pi_d)^{-1}(G)\cap H_{d,d}=(\pi_d)^{-1}(G_d)$, the map $\pi_d$ induces a homeomorphism $K_d\to H_d$ (Proposition~\ref{Prop:DegreD}), hence $K_d$ is irreducible. Its closure  $K= \overline{K_d}$ is contained in the closed set $(\pi_d)^{-1}(G)$, hence satisfies $\pi_d(K)\subseteq G$. Furthermore, on the one hand, $\pi_d(K)$ is closed (by Corollary~\ref{Coro:QuotientBirPnd}) and on the other hand $\pi_d(K)$ contains the dense open set $G_d$ of $G$. Therefore,  $\pi_d(K)=G$.

Let us fix some element $f \in K$ and consider the map $K \to H_{d^2}$, $g \mapsto g \circ f$. We claim that there exists $p \in \kk [x_0, \ldots, x_n]_{d^2-d}$, a homogeneous polynomial of degree $d^2-d$ which divides each component of each $g \circ f$, $g\in K$.
Indeed, for any $g \in K$, the transformation $\pi_{d^2}(g \circ f)$ belongs to $G$, so has  degree $\le d$. If this degree is exactly $d$, there exist $p_g \in \kk [x_0, \ldots, x_n]_{d^2-d}$, $h_g \in H_d$ such that   $g \circ f = p_g h_g$. Supposing that $\pi_d(g)$ has degree $d$, the hypersurface $p_g=0$ is contracted by $f$ onto the base-points of $g$, so $p_g$ is a product of some factors of ${\rm Jac}\,f$ (the Jacobian determinant of $f$). This gives, up to multiplicative constants, finitely many polynomials of degree $d^2-d$, which we denote by $p_1,\ldots,p_s$. For $i=1, \ldots,s$, let $W_i$ be the set of elements $g \in K$ such that $p_i$ divides each component of $g \circ f$. Each $W_i$ is closed in $K$, and $W_1 \cup \ldots \cup W_s$ contains the open subset of elements $g$ such that $\pi_d(g)$ and $\pi_{d^2}(g\circ f)$ have degree $d$. Since $K=\overline{K_d}$ is irreducible, we get $K = W_i$ for some $i$ and our claim is proved.

Consequently, there exists a morphism $\mu_f \colon K \to (\pi_d)^{-1}(G)\subseteq H_d$ such that for each $g \in K$, $g \circ f = p \, \mu_f (g)$. Since $U=(\mu_f)^{-1}(K_d)=(\mu_f)^{-1}(H_{d,d})$ is a dense open subset of $K$, we obtain $\mu_f(K)=\mu_f(\overline{U})\subseteq \overline{\mu_f(U)}\subseteq K$.

Since $\pi_d(K)=G$, there exists $g \in K$ such that $\pi_{d^2}(f \circ g) = \pi_{d^2}(g \circ f)= {\rm id}_{\p^n}$. As a consequence, the morphisms $m_g \circ m_f$ and $m_f \circ m_g \colon K \to K$ coincide with the identity morphism on $K_d$, which is dense in $K$, so  they are equal to the identity on $K$. This shows that $m_f\colon K\to K$ is an isomorphism.

We are now ready to prove (1). Let us prove that if $\pi_d (g) =\pi_d (h)$ for some $g,h \in K$, then $g=h$. Choose $f \in K$ such $\pi_d(g) \circ \pi_d(f) \in G_d$.  Since $\pi_d ( m_f (g) ) = \pi_d ( m_f (h) ) \in G_d$, this implies that
$ m_f (g)  =  m_f (h)$, whence $g=h$. We know that $\pi_d$ induces a bijective morphism $K \to G$. By Corollary~\ref{Coro:QuotientBirPnd} this map is closed, so it is a homeomorphism.

(2)  If  $\rho \colon A \to K$ is a morphism of algebraic varieties, then $\pi_d \circ \rho$ is a morphism $A \to \Bir(\p^n)$ having its image in $G$. Conversely, let $\varphi \colon A \to \Bir(\p^n)$ be a morphism having its image in $G$. It corresponds to a unique map $\rho \colon A \to K$ such that $\varphi=\pi_d \circ \rho$. By Proposition~ \ref{Prop:DegreD},  $\rho$ is a morphism of algebraic varieties on the open set $\varphi^{-1} (G_d)$. For any $f\in K$, we compose $\rho$ and $\varphi$ with $m_f$ and the right-multiplication by $\pi_d(f)$, and use Proposition~\ref{Prop:DegreD} to see that $\rho$ is a morphism on $\rho^{-1}((m_f)^{-1}(K_d))$. Since the open sets $(m_f)^{-1}(K_d)$ cover $K$, $\rho$ is a morphism from $A$ to $K$.

$(3)$ Since the map $\pi_d\colon K\to \Bir(\p^n)$ is a morphism in the sense of Definition~\ref{Defi:Family}, so also is $I \circ \pi_d \colon K\to \Bir(\p^n)$, where $I\colon \Bir(\p^n)\to \Bir(\p^n)$ is the map which sends an element on its inverse. By $(2)$, the morphism $I\circ \pi_d\colon K\to G$ corresponds to a morphism of algebraic varieties $K\to K$, which is the inverse map of the group structure. In the same way, the composition of the morphism (of algebraic varieties) $K \times K \to H_{d^2}$, $(f,g) \mapsto f \circ g$ and of
the morphism (in the sense of Definition~\ref{Defi:Family}) $\pi_{d^2} \colon H_{d^2} \to \Bir(\p^n)$ (see Lemma~\ref{lem:WHalgebraic}) gives a morphism $K \times K \to 
\Bir(\p^n)$ having image in $G$. By (2), we obtain a morphism $K \times K \to K$ which is the product map of the group structure.
\end{proof}

\begin{corollary} \label{Cor:AlgGnotconnected}
Let $G\subseteq \Bir(\p^n)$ be a subgroup, closed for the Zariski topology and of bounded degree.
There exist an algebraic group $K$, together with a morphism $K\to \Bir(\p^n)$ inducing a homeomorphism $\pi\colon K\to G$, which is a group homomorphism, and such that for any irreducible algebraic variety $A$, the morphisms $A\to \Bir(\p^n)$  $($in the sense of Definition~$\ref{Defi:Family})$ having their image in $G$
correspond, via $\pi$, to the morphisms of algebraic varieties $A\to K$.
\end{corollary}

\begin{proof}
The fact that  $G$ has a finite number of irreducible components can be checked in the same way as at the beginning of the proof of Proposition~\ref{Prop:ClosedSubgroups}. As for algebraic groups, one can successively establish the following points: 
(1) exactly one irreducible component of $G$ passes through ${\rm id}$; (2)  this irreducible component is a (closed) normal subgroup of finite index in $G$, whose cosets are the connected as well as irreducible components of $G$ (see \cite[\S7.3]{Humphreys}). The result  then follows from the connected case, treated in Proposition~\ref{Prop:ClosedSubgroups}.
\end{proof}

\begin{lemma}\label{Lemma:AlgGrNonInt}
Let $A$ be an algebraic group, and let $\rho\colon A\to \Bir(\p^n)$ be a morphism $($in the sense of Definition~$\ref{Defi:Family})$, which is also a group homomorphism. Then, the image $G$ of $A$ is a closed subgroup of $\Bir(\p^n)$, which has bounded degree. Denoting by $\pi \colon K\to G$ the homeomorphism constructed in Corollary~$\ref{Cor:AlgGnotconnected}$, there is a unique morphism of algebraic groups $\tilde{\rho}\colon A\to K$ such that $\rho = \pi  \circ \tilde{\rho}$.
\end{lemma}

\begin{proof}
By Lemma~\ref{Lem:ReleveLocalBirPn}, the image $G=\rho(A)$ is of bounded degree. Denote by $\overline{G}$ its closure, and observe that it is a subgroup of $\Bir(\p^n)$ (see the proof of \cite[Proposition A, page 54]{Humphreys}).

By Corollary~\ref{Cor:AlgGnotconnected}, we obtain a canonical homeomorphism $K\to \overline{G}$, where $K$ is an algebraic group, and a lift $\tilde{\rho}\colon A\to K$ of the morphism $\rho\colon A\to \Bir(\p^n)$ having image in $\overline{G}$. Since $\rho$ is a group homomorphism, $\tilde{\rho}$ is a morphism of algebraic groups, hence its image is closed, so it is $K$ itself. This shows that $\overline{G}=G$ and completes the proof.
\end{proof}

\begin{remark}
In the literature, an algebraic subgroup $G$ of $\Bir(X)$ corresponds to taking an algebraic group $G$ and a morphism $G\to \Bir(X)$ (in the sense of Definition~\ref{Defi:Family}) which is a group homomorphism, and whose schematic kernel is trivial (which is not the case of the action of the additive group on $\A^1$ $(a,x)\mapsto (x+a^p)$ in characteristic~$p$).

Corollary~\ref{Cor:AlgGnotconnected} allows one to give a more intrinsic definition in the case of $X=\p^n$, which corresponds to taking closed subgroups of $\Bir(\p^n)$ of bounded degree, and Lemma~\ref{Lemma:AlgGrNonInt} shows that the two definitions agree.
\end{remark}

\begin{remark}
One can show that any algebraic subgroup $G$ of  $\Bir( \p^n)$ is linear, and this reduces to the connected case. By a classical result of Weil, $G$ acts by automorphisms on some (smooth) rational variety $X$. Denote by $\alpha_X \colon X \to A(X)$ the Albanese morphism, that is the universal morphism to an abelian variety. Then, the Nishi-Matsumura theorem asserts that the induced action of $G$ on $A(X)$ factors through a morphism $A(G) \to A(X)$ with finite kernel (see \cite{Brion}).  However, since $X$ is rational, $A(X)$, and then $A(G)$ are trivial. Hence, $G$ is affine by the structure theorem of Chevalley.
\end{remark}

\section{A crucial example, and the proof of Theorem~\ref{ThmNoStructure}}
The following example will be essential in the proofs of Theorems~\ref{ThmNoStructure} and \ref{thmZar}.

\begin{example}\label{TheExample}

Let $\hat{V}$ be $\p^2\backslash\{(0:1:0),(0:0:1)\}$ and let $\tilde{\rho}\colon \hat{V}\to H_2$ be the morphism which sends $(a:b:c)$ onto
$$(x_0(ax_2+cx_0):x_1(ax_2+bx_0):x_2(ax_2+cx_0):\dots :x_n(ax_2+cx_0)),$$
and define $V\subseteq \Bir(\p^n)_{\le 2}$ to be the image of the  morphism $\rho=\pi_2\circ \tilde{\rho}\colon \hat{V}\to \Bir(\p^n)$.
The map ${\rho}\colon\hat{V}\to V$ sends the line $L\subseteq \hat{V}$ corresponding to $b=c$ to the identity,
and induces a bijection $\hat{V}\backslash L \to V\backslash \{\id\}$.
\end{example}
\begin{remark}
The above map corresponds in affine coordinates to
$$\begin{array}{rccl} \hat{V}\times \A^n&\dasharrow& \A^n,\\
\left((a:b:c),(x_1,\dots,x_n)\right)&\dasharrow& (x_1\cdot \frac{ax_2+b}{ax_2+c},x_2,\dots,x_n).\end{array}$$
\end{remark}

\begin{lemma}\label{Lem:PropertiesOfExample}
The algebraic varieties and maps of Example~$\ref{TheExample}$ satisfy:
\begin{enumerate}
\item
The set $V\subseteq \Bir(\p^n)$ is closed.
\item
The map $\rho\colon\hat{V}\to V$ is a topological quotient map.
\end{enumerate}
\end{lemma}

\begin{proof}For $(1)$, we show that $\tilde{\rho}(\hat{V})$ is closed in $H_2$ (recall that $\pi_2$ is closed by Corollary~\ref{Coro:QuotientBirPnd}).
Let $A\subseteq H_2$ be the closed set of elements $(h_0:\dots:h_n)\in H_2$ satisfying $h_0 x_j=h_j x_0$ for $j=2,\dots,n$. Each element of $A$ is of the form $(x_0 l_1:l_2:x_2 l_1:\dots:x_n l_1)$, where $l_i$ is a form of degree $i$. Imposing the conditions that $l_1$  depends  only on $x_0$ and $x_2$ and that $l_2$ is a linear combination of $x_1x_2$ and $x_0x_1$, we get  
$$(x_0(ax_2+cx_0):x_1(dx_2+bx_0):x_2(ax_2+cx_0):\dots :x_n(ax_2+cx_0))$$
for some $a,b,c,d\in \kk$ satisfying $(b,d)\not=(0,0)$ and $(a,c)\not=(0,0)$. Adding the condition that $a=d$ yields the closed set $\tilde{\rho}(\hat{V})$.

$(2)$ The description of $\tilde{\rho}(\hat{V})$ shows that $\tilde{\rho}$ is a closed embedding. Since $\pi_2$ is closed and continous  (Corollary~\ref{Coro:QuotientBirPnd}), so is $\rho$. 
\end{proof}
The following result implies Theorem~\ref{ThmNoStructure}.

\begin{proposition}\label{Prop:led}
For $n,d\ge 2$, the following hold:
\begin{enumerate}
\item
There is no structure of algebraic variety $($or ind-algebraic variety$)$ on $\Bir(\p^n)_{\le d}$, such that algebraic $A$-families $A\to \Bir(\p^n)$ having image in $\Bir(\p^n)_{\le d}$ correspond to morphisms of $($ind$)$-algebraic varieties $A\to \Bir(\p^n)_{\le d}$.
\item
There is no structure of ind-algebraic variety $($or algebraic variety$)$ on $\Bir(\p^n)$, such that algebraic $A$-families $A\to \Bir(\p^n)$ correspond to the morphisms of $($ind$)$-algebraic varieties $A\to \Bir(\p^n)$.
\end{enumerate}
\end{proposition}

\begin{proof}Suppose, for contradiction, the existence of the ind-algebraic variety structure on $\Bir(\p^n)_{\le d}$ or $\Bir(\p^n)$. 
Let $\rho\colon \hat{V}\to \Bir(\p^n)$ be the morphism given by Example~\ref{TheExample}. Its image is   $V\subseteq\Bir(\p^n)_{\le 2}\subseteq \Bir(\p^n)_{\le d}$, which is closed in $\Bir(\p^n)$, hence in $\Bir(\p^n)_{\le d}$ (Lemma~\ref{Lem:PropertiesOfExample}). The map $\rho$  thus corresponds to a morphism of ind-algebraic varieties from $\hat{V}$ to  $\Bir(\p^n)_{\le d}$ or $\Bir(\p^n)$. Since $\hat{V}$ is an algebraic variety, the morphism factors through a morphism from $\hat{V}$ to a closed algebraic variety (of finite dimension). Since the image $V$ of $\hat{V}$ is closed, the map $\hat{V}\to V$ is then a morphism of algebraic varieties,  so $V$ is an irreducible variety of dimension~$2$. Lemma~\ref{Lem:PropertiesOfExample} asserts that this map is a topological quotient map. Hence, all closed sets of $V$ correspond  either to points, to $V$ itself or to images of curves of $\hat{V}$. In particular, all curves of $V$ pass through the same point, which is impossible for an algebraic variety of dimension $\ge 2$.
\end{proof}

\begin{remark}
In the above proof, the fact that $\rho\colon \hat{V}\to V$ cannot be a morphism of algebraic varieties can be seen in another way. Extending to $\p^2$, we would obtain a birational map $\p^2\dasharrow V$, defined on the dense open subset $\hat{V}$, which contracts the line $L\subseteq \hat{V}$ onto a point. This is impossible, since the line has positive self-intersection and does not contain any base-point. 
It can also be verified that $\rho\colon \hat{V}\to V$ is not a morphism of algebraic spaces or of algebraic stacks.
\end{remark}

\section{Topological obstructions to being an ind-algebraic group}\label{Sec:TopObs}
To distinguish the topology on $\Cr_n(\kk)$ from those of algebraic varieties, we define the following notion (note that the fact that $\kk$ is algebraically closed is important here).

\begin{definition}
Let $S$ be a topological space.  We say that a closed point $p\in S$ is an \emph{attractive point of $S$} if $S$ contains an infinite proper closed subset, and if $p$ is contained in every infinite closed set $F\subseteq S$.
\end{definition}

Recall that an ind-algebraic variety $Y$ corresponds to a sequence $Y_1\subseteq Y_2\subseteq Y_3\subseteq \dots$ of algebraic varieties, each closed in the next one. There are two classical ways of putting a topology on $Y$. The first, introduced by Shafarevich \cite{Sha2} corresponds to saying that a set $F\subseteq Y$ is closed if  $F\cap Y_i$ is closed in $Y_i$ for each~$i$. The second, introduced by Kambayashi \cite{Kam1,Kam2}, is defined by taking zero sets of functions obtained by projective limits of functions on the $Y_i$, when all the $Y_i$ are affine (if the $Y_i$ are not affine we can define the topology on sequences of affine subvarieties as is done in the usual way for varieties). Following \cite{Sta}, we  call the first one the ind-topology and the second one the Zariski topology. Note that the two topologies are obviously the same on algebraic varieties;  however they differ on most ind-algebraic varieties (see \cite{Sta}).

\begin{lemma}\label{attractive points of a closed irreducible subset of an ind-variety}
Let $X$ be an ind-algebraic variety. Putting the Zariski or ind-topology on $X$, the set of attractive points of any closed irreducible subset of $X$ is empty. 
\end{lemma}
\begin{proof}
We let $Y\subseteq X$ be an irreducible closed subset, and want to prove that the topological space given by the points of $Y$ does not contain any attractive points.

Assume first that $Y$ is an algebraic variety of  dimension $\ge 2$.
Then, for any point $p\in Y$, we can find a closed curve $\Gamma\subseteq Y$ that does not pass through $p$.
Consequently, $Y$ does not contain any attractive points.
The same holds if there is a closed subset $Z\subseteq Y$ which is an algebraic variety of dimension $\ge 2$.
If $Y$ is a point or an irreducible curve, it does not contain any infinite closed proper subset, so does not contain any attractive points by definition.

There remains to study the case in which $Y$ is an ind-variety, being the union $Y_1\subseteq Y_2\subseteq Y_3\subseteq \dots$, where all $Y_i$ are algebraic varieties of dimension $\le 1$ and where $Y_i\not= Y_{i+1}$ for each $i$. We take a point $p\in Y$, and denote by $k$ an integer such that $p\in Y_k$. To obtain the result, we construct a closed set $F\subseteq Y$ that is infinite and does not contain $p$.

If $Y$ is endowed with the ind-topology,  we choose $q_i \in Y_i\backslash Y_{i-1}$ for any $i\ge k+1$,
and let $F$ be the union of all $q_i$.

If $Y$ is endowed with the Zariski topology, we take a sequence $Z_1\subseteq Z_2 \subseteq \dots $ of affine varieties of dimension $\le 1$,
where $Z_i$ is open and dense in $Y_i$ for each $i$; this corresponds to removing a finite number of points from each non-affine curve in $Y_i$. We can choose $p\in Z_k$ and $Z_i\not= Z_{i+1}$ for each $i$. For $i=1,\dots,k$ let $f_i\in \kk[Z_{i}]$  be the constant function $1$. Then, for each $i\ge k+1$ we choose a function $f_i\in \kk[Z_{i}]$ that restricts to $f_{i-1}$ on $Z_{i-1}$ and vanishes at some point of $Z_i\backslash Z_{i-1}$. The vanishing set $F'$ of the $(f_i)_{i\in \mathbb{N}}$ is infinite, closed in the union of the $Z_i$ and does not contain $p$. The set $F$ which we require can be chosen  as the closure of $F'$ in~$Y$.
\end{proof}
\begin{remark}
Let us note that the second case considered in the proof of Lemma \ref{attractive points of a closed irreducible subset of an ind-variety} can actually occur, in the sense that $Y$might actually be irreducible! Indeed, consider the ind-variety $Y$ given by the filtration $(Y_n)_{n \geq 1}$, where $Y_n$ is the zero set of $(x-1)(x-2) \ldots (x-n)(y-1)(y-2) \ldots (y-n)$ in the complex affine plane $\A^2_{\C} = {\rm Spec}\, \C[x,y]$; then $Y$ is irreducible for the ind-topology (and hence for the Zariski topology, which is weaker).
\end{remark}

The following result implies Theorem~\ref{thmZar}.

\begin{proposition}\label{Prop:attractive}
If $n\ge 2$, any element $\varphi\in \Cr_n(\kk)$ is an attractive point of an irreducible closed subset  $Y\subseteq \Cr_n(\kk)$.
\end{proposition}
\begin{proof}
We may assume that $\varphi$ is the identity, since the map $\Cr_n(\kk)\to \Cr_n(\kk)$
given by $\varphi'\mapsto \varphi\circ \varphi'$ is a homeomorphism.
We use the map $\rho\colon \hat{V}\to V\subseteq \Bir(\p^n)\cong \Cr_n(\kk)$ of Example~\ref{TheExample}. By Lemma~\ref{Lem:PropertiesOfExample}, $V$ is closed in $\Cr_n(\kk)$, and its topology is given by the quotient $\hat{V}\to V$. The set $\hat{V}$ is irreducible, contains infinitely many infinite closed subsets
(the irreducible ones corresponding to itself and curves),
and each intersects $L=(\rho)^{-1}(\mathrm{id})$.
This shows that $V$ is irreducible, contains infinitely many infinite closed subsets, and that each contains the identity.
\end{proof}

\section{Euclidean topology} \label{Sec:Trans}
On the points of a real or complex algebraic variety, we can put the Euclidean (also called transcendental) topology, which is finer than the Zariski topology. This gives any algebraic group the structure of a topological group. We  imitate this in this section, working over a local field $\k$. We will assume that $\k$ is locally compact and nondiscrete. Particular cases are $\C$, $\R$, $\mathbb{F}_q((t))$ or a  finite extension of $\mathbb{Q}_p$.

We fix some integer $n$, which will be the dimension, and we will prove that $\Cr_n(\k)$ is a topological group,
endowed with the \emph{Euclidean topology} which we will define. Note that the set of $\tilde{\k}$-points, where $\tilde{\k}$ is a subfield of $\k$, will also inherit the structure of a topological group, such as for example $\Cr_n(\mathbb{Q})$.\\

For $n=1$, the group $\Cr_n(\k)=\Aut(\p^1_\k)=\PGL(2,\k)$ is obviously a topological group, so we will  deal only with the cases in which $n\ge 2$. 

In subsection \ref{The Euclidean topology on Bir(p^n)_{le d}}, we define the Euclidean topology on  $\Bir(\p^n)_{\le d}$ and show that the natural inclusion $\Bir(\p^n)_{\le d}\hookrightarrow \Bir(\p^n)_{\le d+1}$ is 
a closed embedding. This allows us, in subsection \ref{The Euclidean topology on the Cremona group}, to define the Euclidean topology on $\Bir(\p^n)$ as the inductive limit topology induced by those of $\Bir(\p^n)_{\le d}$: a subset $F\subseteq \Bir(\p^n)$ is closed if and only if $F\cap \Bir(\p^n)_{\le d}$ is closed in $\Bir(\p^n)_{\le d}$ for each $d$.

\subsection{\boldmath The Euclidean topology on $\Bir(\p^n)_{\le d}$}
\label{The Euclidean topology on Bir(p^n)_{le d}}

We  use the notation  $W_d,H_d$ as in Definition~$\ref{DefWHG}$; these varieties are defined over any field. By Lemma~\ref{lem:WHalgebraic}, $W_d$ is a projective space and $H_d$ is locally closed in $W_d$ for the Zariski topology. 

We can thus put the Euclidean topology on the projective space $W_d$. For example, we say that the distance between $(x_0:\dots:x_r)$ and $(y_0:\dots:y_r)$ is equal to
$$\left(\sum_{i<j} |x_iy_j-y_jx_i|^2\right)/\left(\left(\sum_{i} |x_i|^2\right)\cdot \left(\sum_{i} |y_i|^2\right)\right)$$
(see \cite{Weyl}). We then put the induced topology on $H_d$.
Because of the behaviour of the Zariski topology of $\Bir(\p^n)$, it is natural to give the following definition:

\begin{definition}
The Euclidean topology on $\Bir(\p^n)_{\le d}$ will be the quotient topology induced by the surjective map $\pi_d\colon H_d\to \Bir(\p^n)_{\le d}$, where we put the Euclidean topology on $H_d$.
\end{definition}

If $f \colon X \to Y$ is a quotient map between topological spaces and  $A$ is a subspace of $X$, note that the induced map $A \to f(A)$ is not always a quotient map. However, this becomes true, if $A$ is open and $A= f^{-1}(f(A))$ \cite[I, \S3.6, Corollary~1]{Bou}. Since $H_{d,d}=(\pi_d)^{-1}(\Bir(\p^n)_{d})$ is open in $H_d$ for the Zariski topology (see Proposition~\ref{Prop:DegreD}) and hence also for the Euclidean topology,  $\pi_d$ restricts to a homeomorphism $(\pi_d)^{-1}(\Bir(\p^n)_{d})\to \Bir(\p^n)_{d}$, for any $d\ge 1$.

\begin{lemma} \label{Lemm:Sequential}
Let $d\ge 1$ be an integer. Endowed with the Euclidean topology, $W_d$ and $H_d$ are locally compact metric spaces. In particular, the sets $W_d,H_d,\Bir(\p^n)_{\le d}$ are sequential spaces: a subset $F$ is closed if the limit of every convergent sequence with values in $F$ belongs to $F$.
\end{lemma}

\begin{proof}
By construction of the topology, $W_d$ and $H_d$ are metric spaces. Since $W_d$ is compact and $H_d$ is locally closed in $W_d$ (Lemma~\ref{lem:WHalgebraic}), $H_d$ is locally compact.
It remains to recall that metric spaces are sequential spaces and that quotients of sequential spaces are sequential (see \cite{Fra}).
\end{proof}

\begin{remark}
Following \cite{Bou}, we say that a map $f:X \to Y$ between two topological
spaces is \emph{proper} if it is continuous and universally closed
(i.e. for each topological space $Z$, the map
$f \times {\rm id}_Z : X \times Z \to Y \times Z$ is closed).
We also say that a topological space is locally compact if it is Hausdorff and if each of its points has a compact neighbourhood.
\end{remark}

\begin{lemma}\label{Lem:Closed}
For any $d\ge 1$, the following hold:
\begin{enumerate}
\item
the topological map $\pi_d\colon H_d\to \Bir(\p^n)_{\le d}$ is proper $($and closed$)$;
\item
the topological space $\Bir(\p^n)_{\le d}$ is locally compact $($and Hausdorff$)$.
\end{enumerate}
\end{lemma}

\begin{proof}
We recall the following  result of topology \cite[I, \S10.4, Proposition~9]{Bou}: if $f\colon X\to Y$ is a quotient map between topological spaces such that $X$ is locally compact,
then $f$ is proper if and only if it is closed and the preimages of points are compact.
This implies moreover that $Y$ is locally compact.

The fact that $(\pi_d)^{-1}(\varphi)$ is compact for any $\varphi\in \Bir(\p^n)_{\le d}$ follows from Lemma~\ref{lem:WHalgebraic}, since $(\pi_d)^{-1}(\varphi)$ is closed in the compact space $W_d$. Since the topological space $H_d$ is locally compact (Lemma~\ref{Lemm:Sequential}), it suffices to see that $\pi_d$ is closed.

We let $F\subseteq H_d$ be a closed subset, and want to prove that $\pi_d(F)$ is closed in $\Bir(\p^n)_{\le d}$, which amounts to showing that the saturated set $\hat{F}=(\pi_d)^{-1}(\pi_d(F))$ is closed in $H_d$. For this, we take a sequence $(\varphi_i)_{i\in \mathbb{N}}$ of elements in $\hat{F}$, which converges to $\varphi\in H_d$, and show that $\varphi\in \hat{F}$. 

Since the map $\pi_d$ is continuous by construction, the sequence $(\pi_d(\varphi_i))_{i\in \mathbb{N}}$ converges to $\pi_d(\varphi)$ in $\Bir(\p^n)_{\le d}$.
Replacing $(\varphi_i)_{i\in \mathbb{N}}$ by a subsequence, we may assume that the degree of all $(\pi_d(\varphi_i))_{i\in \mathbb{N}}$ is constant, equal to $m\le d$. If $m=d$, then $(\pi_d)^{-1}(\pi_d(\varphi_i))=\{\varphi_i\}$ for each $i$, so each $\varphi_i$ belongs to $F$, which implies that $\varphi\in F\subseteq \hat{F}$ and completes the proof. 
We may thus assume that $m<d$ (and hence that $d\ge 2$), and denote by $k$ the difference $d-m\ge 1$. For each $i$, there exists a non-zero homogeneous polynomial $a_i\in\k[x_0,\dots,x_n]$ of degree $k$ such that 
$$\varphi_i=(a_i f_{i,0}:\dots:a_i f_{i,n}),$$
and $(f_{i,0}:\dots: f_{i,n})\in W_{m}$ corresponds to a birational map of degree $m< d$. 
Since each $a_i$ is defined up to a constant, and since $\p(\k[x_0,\dots,x_n])$ is compact, we can take a subsequence and assume that  $(a_i)_{i\in \mathbb{N}}$ converges to a non-zero homogeneous polynomial $a\in\k[x_0,\dots,x_n]$ of degree $k$.

We can then again take a subsequence and assume that $\{(f_{i,0}:\dots: f_{i,n})\}_{i\in \N}$
converges to an element $(f_0:\dots:f_n)$  of the set $W_{m}$, which is also a projective space. Since $(\varphi_i)_{i\in \mathbb{N}}$ converges to $\varphi$, we get
$\varphi=(af_0:\dots:af_n)$ in $H_{d}$.

For each $i$,  there exists $\varphi_i'\in F$ with $\pi_d(\varphi_i')=\pi_d(\varphi_i)$ (because $\varphi_i\in \hat{F}=(\pi_d)^{-1}(\pi_d(F))$). We then have
$\varphi_i'=(b_i f_{i,0}:\dots:b_i f_{i,n}),$
for some non-zero homogeneous polynomial $b_i\in\k[x_0,\dots,x_n]$ of degree $k$. We may again assume that the sequence $(b_i)_{i\in \mathbb{N}}$ converges to a non-zero homogeneous polynomial $b\in\k[x_0,\dots,x_n]$ of degree $k$. In particular, the sequence $(\varphi_i')_{i\in \mathbb{N}}$ converges to $(bf_0:\dots:bf_n)$, which is in $F$ since $F$ is closed. This implies that $\varphi=(af_0:\dots:af_n)$ belongs to~$\hat{F}$.
\end{proof}

\begin{remark}
For any $d\ge 2$, the map $\pi_d\colon H_d\to \Bir(\p^n)_{\le d}$ is not open. To see this, we define 
$f_m=(x_0x_2^{d-1}:x_1(x_2^{d-1}+\frac{1}{m}x_0^{d-1}):x_2^{d}:x_3x_2^{d-1}:\dots:x_nx_2^{d-1})\in H_d$ for  $m\ge 1$.
Since $\{f_m\}_{m\in \mathbb{N}}$ converges in $H_d$ to $f_\infty \in \pi_d^{-1}(\mathrm{id})$,  we get a sequence $\{\pi_d(f_m)\}_{m\in \mathbb{N}}$ of elements of degree $d$ converging to the identity. 
Let $g$ be any element of  $\pi_d^{-1}(\mathrm{id})$  such that $g \neq f_{\infty}$ and let $U \subseteq H_d$ be an open neighbourhood of $g$  not containing any of the $f_m$. Then,  $\pi_d(U)$ is not open since $\{\pi_d(f_m)\}_{m\in \mathbb{N}}$ converges to $\mathrm{id} \in \pi_d(U)$ but $\pi_d(U)$ does not contain any of the $\pi_d(f_m)$.\end{remark}

\begin{lemma}\label{Lem:Embbeddings}For any positive integer $d$,
the natural injection $\iota_d\colon \Bir(\p^n)_{\le d}\hookrightarrow \Bir(\p^n)_{\le d+1}$ is 
a closed embedding, i.e.\ a homeomorphism onto its image, which is closed in $\Bir(\p^n)_{\le d+1}$.
\end{lemma}

\begin{proof}
We define a map $\widehat{\iota_d}\colon H_d \to H_{d+1}$ by $\widehat{\iota_d} ((f_0:\dots:f_n))=(x_0f_0:\dots:x_0f_n)$.
It is a morphism of algebraic varieties, which is a closed immersion; it is thus continuous and closed with respect to the Euclidean topology. We have the following commutative diagram: 
$$\xymatrix@R=4mm@C=2cm{H_d\ar[d]^{\pi_d}\ar[r]^{\widehat{\iota_d}}& H_{d+1}\ar[d]^{\pi_{d+1}}\\
\Bir(\p^n)_{\le d}\ar[r]^{{\iota_d}}& \Bir(\p^n)_{\le d+1}.
}$$

$a)$ The continuity of $\widehat{\iota_d}$ directly implies the continuity of  $\iota_d$. Indeed, if $U$ is an open subset of $\Bir(\p^n)_{\le d+1}$,  the equality  $(\pi_d)^{-1}((\iota_d)^{-1}(U)) =(\pi_{d+1}\widehat{\iota_d})^{-1}(U)$ shows that  $(\pi_d)^{-1}((\iota_d)^{-1}(U))$ is open in $H_d$, i.e.\ that  $(\iota_d)^{-1}(U)$ is open in 
$\Bir(\p^n)_{\le d}$.

$b)$  It is clear that $\iota_d$ is injective. We only need to prove that it is closed. Since $\pi_{d+1}$ and $\hat{\iota_d}$ are closed, so is $\pi_{d+1}\circ\hat{\iota_d}=\iota_d \circ \pi_d$.
Because $\pi_d$ is continuous and surjective, this implies that $\iota_d$ is closed.
\end{proof}

\subsection{The Euclidean topology on the Cremona group} \label{The Euclidean topology on the Cremona group}

Using Lemma~\ref{Lem:Embbeddings}, one can put on $\Bir(\p^n)$ the inductive limit topology given by the $\Bir(\p^n)_{\le d}$: a subset of $\Bir(\p^n)$ is closed (respectively open)
if and only if its intersection with each $\Bir(\p^n)_{\le d}$ is closed (respectively open).
In particular, the injections $\Bir(\p^n)_{\le d}\hookrightarrow \Bir(\p^n)$ are closed embeddings.
As explained earlier, the topology defined here is called the \emph{Euclidean topology of  $\Bir(\p^n)$}.

In this subsection we show that $\Bir(\p^n)$, endowed with the Euclidean topology, is a topological group.

\begin{lemma} \label{Lemm:InverseBounded}
For any $d\ge 1$, the map $I_d\colon \Bir(\p^n)_{\le d}\to \G{d^{n-1}}$ which sends an element onto its inverse is continuous.
\end{lemma}

\begin{proof}

As in the proof of Lemma~\ref{lem:WHalgebraic}, we define $Y\subseteq W_{d^{n-1}} \times W_d$ to be the set of elements $(g,f)$ such that $(g_0(f_0,\ldots,f_n), \ldots,g_n(f_0,\ldots,f_n))$  is a multiple (maybe~$0$) of the identity and let $U \subseteq W_d$ (resp. $U' \subseteq W_{d^{n-1}}$) be the set of elements having a non-zero Jacobian.

As we have already observed, $Y$ is closed in $W_{d^{n-1}} \times W_d $ and $U$ is open in $W_d$,
so that $L:= Y \cap (W_{d^{n-1}} \times U) = Y \cap (U' \times U)$ is locally closed in $W_{d^{n-1}} \times W_d$
(it is an algebraic variety).

The projection on the second factor induces
a surjective morphism $\eta_2 \colon L \to H_d$.
The projection on the first factor is a morphism
$\eta_1 \colon L\to H_{d^{n-1}}$,
which is not surjective in general.
By construction, we have the following commutative diagram:
$$\xymatrix@R=4mm@C=1cm{
H_d\ar[d]^{\pi_{d}}& L\ar[l]_{\eta_2}\ar[r]^{\eta_1}&H_{d^{n-1}}\ar[d]^{\pi_{{d^{n-1}}}}\\
\Bir(\p^n)_{\le d}\ar[rr]^{I_{d}}&& \G{{d^{n-1}}}.
}$$

We claim that $\eta_2$ is a closed map, for the Euclidean topology.
Since  $W_{d^{n-1}}$ is compact, the second projection  $ W_{d^{n-1}} \times W_d \to W_d$ is a closed map. Its restriction to the closed subset $Y \subseteq  W_{d^{n-1}} \times W_d$ obviously yields a closed map $\eta'_2 : Y \to W_d$. Finally, let us recall that if $\varphi: A \to B$ is any continuous closed map between topological spaces and if $C$ is any subset of $B$, then $\varphi$ induces a continous closed map $\varphi ^{-1}(C) \to C$.
Therefore, since $L= (\eta'_2)^{-1}(H_d)$, the claim is proved.

For any subset $F$ of $\G{{d^{n-1}}}$, we have $\eta_2((\pi_{d^{n-1}}\eta_1)^{-1}(F))=(I_d\pi_d)^{-1}(F)$.
Indeed, both sets correspond to elements $(f_0:\dots:f_n)\in W_d$
such that the rational map $\psi_f$ is the inverse of an element of $F$.

Assume that $F$ is closed in $\G{{d^{n-1}}}$.
Since $\eta_1$ and $\pi_{d^{n-1}}$ are continuous for the Euclidean topology,
the set $F_L=(\pi_{d^{n-1}}\eta_1)^{-1}(F)$ is closed in $L$. Therefore $\pi_d^{-1}(I_d^{-1}(F))=(I_d\pi_d)^{-1}(F)=\eta_2(F_L)$ is closed in $H_d$, which implies that $I_d^{-1}(F)$ is closed in $\Bir(\p^n)_{\le d}$.
\end{proof}

\begin{corollary} \label{Cor:InverseBounded}
The map $I\colon \Bir(\p^n)\to \Bir(\p^n)$ which sends a map onto its inverse is a homeomorphism.
\end{corollary}

\begin{proof}
Because $I$ is its own inverse, we only need to prove that $I$ is continuous.
As we mentioned before, the degree of the inverse of a birational transformation of $\p^n$ of degree $d$ has degree at most $d^{n-1}$.
Consequently, $I$ restricts to an injective map $I_d\colon \Bir(\p^n)_{\le d}\to \G{d^{n-1}}$, for any $d\ge 1$.
Because of the definition of the topology of $\Bir(\p^n)$, it suffices to prove that $I_d$ is continuous for each $d$. This follows from Lemma~\ref{Lemm:InverseBounded}.
\end{proof}
\begin{lemma}\label{Lem:ProductBounded}
For any $d,k$, the map $\chi_{d,k}\colon \Bir(\p^n)_{\le d}\times \G{k}\to \G{dk}$ which sends $(\varphi_1,\varphi_2)$
onto $\varphi_1\circ \varphi_2$ is continuous.
\end{lemma}
\begin{proof}
We use the following commutative diagram:
$$\xymatrix@R=10mm@C=1.5cm{
H_d\times H_k\ar[d]^{\pi_{d}\times \pi_k}\ar[r]^{\hat{\chi}_{d,k}}& H_{dk}\ar[d]^{\pi_{dk}}\\
\Bir(\p^n)_{\le d}\times \G{k}\ar[r]^{\chi_{d,k}}& \G{dk},
}$$
where ${\hat{\chi}_{d,k}}$ sends $((f_0:\dots:f_n),(g_0:\dots:g_n))$ onto $(f_0(g_0,\dots,g_n):\dots:f_n(g_0,\dots,g_n))$.
Since ${\hat{\chi}_{d,k}}$ is a morphism of algebraic varieties, it is continuous (for the Euclidean topology).
For any closed subset $F\subseteq \G{dk}$, the set $(\pi_{dk}{\hat{\chi}_{d,k}})^{-1}(F)$ is then closed in $H_d\times H_k$.
Since the diagram is commutative, we have $(\pi_{dk}{\hat{\chi}_{d,k}})^{-1}(F)  =(\pi_d\times \pi_k)^{-1}(K)$,
where $K=({\chi}_{d,k})^{-1}(F)$. It remains to prove that $\pi_{d}\times \pi_k$ is a quotient map, which will imply that $K$ is closed and thus give the continuity of $\chi_{d,k}$.

If $f: X \to Y$ is a quotient map between topological spaces and $Z$ is locally compact, a theorem of Whitehead  \cite[Lemma 4]{White} asserts that
$f \times {\rm id} : X \times Z \to Y \times Z$ is a quotient map.
More generally, if $f:X \to Y$ and $g:Z \to W$ are quotient maps and 
$Y$ and $Z$ are locally compact, then the product $f \times g:X \times Z \to Y \times W$ is a quotient map
(use the Whitehead theorem twice, since $f \times g=({\rm id} \times g) \circ (f \times {\rm id})$).

Alternatively, one can avoid using Whitehead's theorem
by noting that the product of two proper maps is proper
(\cite[I, \S10.1, Proposition~4]{Bou}),
so that $\pi_{d}\times \pi_k$ is proper and hence closed.
This implies that $\pi_{d}\times \pi_k$ is a quotient map.
\end{proof}
\begin{corollary} \label{Cor:ProductBounded}
The map $P:\Bir(\p^n)\times \Bir(\p^n)\to \Bir(\p^n)$ which sends $(\varphi,\varphi')$ onto $\varphi\circ \varphi'$
is continuous $($where we endow $\Bir(\p^n)\times \Bir(\p^n)$ with the product topology of the Euclidean topology
of $\Bir(\p^n_\k))$.
\end{corollary}
\begin{proof}Because of the definition of the topology of $\Bir(\p^n)$, it suffices to prove that the restriction $\chi_{d,k}\colon \Bir(\p^n)_{\le d}\times \G{k}\to \G{dk}$ is continuous for each $d,k$. This follows from Lemma~\ref{Lem:ProductBounded}.
\end{proof}

Corollaries  \ref{Cor:InverseBounded} and \ref{Cor:ProductBounded} complete the proof that $\Bir(\p^n)$ is a topological group.

\subsection{Restriction of the topology on algebraic subgroups}

As we saw in $\S\ref{SubSec:algsubgroups}$, an algebraic subgroup of $\Bir(\p^n)$ corresponds to a Zariski-closed subgroup $G \subseteq \Bir(\p^n)$ of bounded degree; moreover there exists an algebraic group $K$, together with a morphism  $K\to \Bir(\p^n)$ inducing a homeomorphism $\pi\colon K\to G$, which is a group homomorphism (Corollary~\ref{Cor:AlgGnotconnected}).

\begin{proposition} \label{Prop:TopologyAlgSubgroups}
Let $G\subseteq \Bir(\p^n)$ be a Zariski-closed subgroup of bounded degree, and let $K$ be its associated algebraic group $($as in Corollary~$\ref{Cor:AlgGnotconnected})$. Putting on $G$ the restriction of the Euclidean topology of $ \Bir(\p^n)$, we obtain the Euclidean topology of the algebraic group $K$, via the bijection $\pi \colon K\to G$, which becomes a homeomorphism.
\end{proposition}

\begin{proof}
Via the action of $G$ on itself by multiplication, we may restrict ourselves to the case where $G$ is connected. In this case, we may assume that $K$ is a Zariski-closed subset of $H_d$ and that  $\pi \colon K\to G$ is induced by $\pi_d \colon H_d \to  \Bir(\p^n)_{\leq d}$ (see Proposition~\ref{Prop:ClosedSubgroups}). Then, $K$ is also closed for the Euclidean topology. The map $\pi_d$ restricts to a bijection $K\to G$, which is closed and continuous for the Euclidean topology (Lemma~\ref{Lem:Closed}) and is thus a homeomorphism.
\end{proof}

\subsection{Properties of the Euclidean topology of $\Bir(\p^n)$}

\begin{lemma}
The topological group  $\Bir(\p^n)$ is Hausdorff.
\end{lemma}

\begin{proof} 
We recall that a topological group is Hausdorff
if and only if the trivial one-element subgroup is closed (see \cite[III, \S2.5, Proposition 13]{Bou}). 
Since any point of $\Bir(\p^n)$ is closed (being closed in some $\Bir(\p^n)_{\le d}$, by Lemma~\ref{lem:WHalgebraic}),
this implies that $\Bir(\p^n)$ is Hausdorff.
\end{proof}

\begin{lemma} \label{compact sets have bounded degree}
Any compact subset of $\Bir(\p^n)$ is contained in some $\Bir(\p^n)_{\le d}$.
\end{lemma}

\begin{proof}
Assume by contradiction that $K$ is a compact subset of $\Bir(\p^n)$ containing a sequence $(\varphi_i)_{i \in \mathbb{N}}$ with $\deg(\varphi_{i+1})>\deg ( \varphi_i)$ for each $i$.
Since  $K':=\{ \varphi _i\ |\  i \in \N \}$ is a closed subset of the compact set $K$, it should be compact. However, the intersection of any subset of $K'$ with $\Bir(\p^n)_{\le d}$ is closed, so  $K'$ is an infinite set endowed with the discrete topology, hence it cannot be compact.
\end{proof}

\begin{corollary}\label{Coro:SeqDegr}
Any convergent sequence of $\Bir(\p^n)$ has bounded degree.
\end{corollary}

\begin{proof}
Indeed, if the sequence $(\varphi_i)_{i \in \mathbb{N}}$ of $\Bir(\p^n)$
converges to $\varphi$, then $K:= \{ \varphi _i\ |\  i \in \N \} \cup \{ \varphi \}$ is compact.
\end{proof}

\begin{lemma}
For  $n\ge 2$, the topological space $\Bir(\p^n)$ is not locally compact.
\end{lemma}

\begin{proof} 
Let  $U\subseteq \Bir(\p^n)$ be an open neighbourhood of the identity. Let us show that $U$ is not contained in any compact subset of $\Bir(\p^n)$.
By Lemma  \ref{compact sets have bounded degree}, it suffices to show that $U$ contains elements of arbitrarily large degree. For any integers $m,k\ge 1$, we consider the birational map $f_{m,k}$ of $\A^n_\k$ given by
\begin{center}$f_{m,k}\colon (x_1,x_2,\dots,x_n)\dasharrow \left(x_1 + \frac{1}{k} x_2^m,x_2,x_3,\dots,x_n\right).$\end{center}
Fixing $m$, we observe that the sequence $\{f_{m,k}\}_{k \geq 1}$ converges to the identity. In particular, $f_{m,k}$ belongs to $U$ when $k$ is large enough.
\end{proof}

\begin{lemma}
For  $n\ge 2$, the topological space $\Bir(\p^n)$ is not metrisable.
\end{lemma}
\begin{proof}
We consider the set $\k[X]$ of polynomials in one variable and let $\k[X]\hookrightarrow \Aut(\A^n)\subset \Bir(\p^n)$ be the inclusion sending $P$ to
\begin{center}$(x_1,x_2,\dots,x_n)\dasharrow \left(x_1 + P(x_2), x_2, x_3, \dots, x_n \right)$. \end{center}
Note that $\k[X]$ is closed in $\Bir(\p^n)$, and that for any $d$, the induced topology on 
$\k[X]_{\le d}$ is the topology as a vector space (or as an algebraic group). The induced topology on $\k[X]$ is thus the inductive limit topology given by 
\begin{center}$ \k[X]_{\le 1}\subseteq \k[X]_{\le 2}\subseteq \dots$\end{center} For any sequence $l=(l_n)_{n\in \mathbb{N}}$ of positive integers, the set $U_l=\{\sum_{i=0}^d a_i X^i\ |\ |a_i|<1/l_i\}$ is open in $\k[X]$. This implies that $\k[X]$ is not first countable and thus not metrisable. The same holds for $\Bir(\p^n)$.
\end{proof}

\begin{lemma}
The topological group $\Bir(\p^n_\mathbb{C})$ is compactly generated if and only if $n\le 2$.
\end{lemma}

\begin{proof}The group $\Bir(\p^1_\mathbb{C})=\PGL(2,\C)$ is a linear algebraic group, hence is compactly generated.
By the classical Noether-Castelnuovo theorem, the group $\Bir(\p^2_\mathbb{C})$ is generated by $\Aut(\p^2_\C)=\PGL(3,\C)$ and by the standard quadratic transformation $\sigma=(x:y:z)\dasharrow (yz:xz:xy)$. Since the linear algebraic group $\Aut(\p^2_\C)=\PGL(3,\C)$ is compactly generated, so is $\Bir(\p^2_\mathbb{C})$.

For $n \geq 3$, the group  $\Bir(\p^n_\mathbb{C})$ is not generated by $\Bir(\p^n_\mathbb{C})_{\le d}$ for any integer~$d$. This  can be seen by showing that the birational type of the hypersurfaces which are contracted by some element of $\Bir(\p^n_\mathbb{C})_{\le d}$ is bounded (see \cite{Pan} for more details). The fact that $\Bir(\p^n_\mathbb{C})$ is not compactly  generated follows from Lemma~\ref{compact sets have bounded degree}.
\end{proof}

\begin{remark}In \cite[Th\'eor\`eme 5.1]{Bl}, it is proved that $\Bir(\p^n_\k)$ is connected for the Zariski topology, when $\k$ is algebraically closed. Looking at the proof, we observe that $\Bir(\p^n_\C)$ and $\Bir(\p^n_\R)$ are in fact path connected (and hence connected) for the Euclidean topology. The same holds for the proof of \cite[Proposition 4.1, Th\'eor\`eme 4.2]{Bl}, which  can be adapted to see that $\Bir(\p^2_\C)$ is a simple topological group for the Euclidean topology (although it is not a simple group \cite{Cantat-Lamy}).
\end{remark}

\begin{remark}
Having the natural Euclidean topology on the sets $\Bir(\p^n)_{\le d}$, there are many ways of extending it to the union of all $\Bir(\p^n)_{\le d}$,
which is the Cremona group $\Bir(\p^n)$; the one that we have chosen is the finest one. It would be interesting to see whether another choice could yield a locally compact topology on $\Bir(\p^n)$, or a metrisable one.
\end{remark}


\begin{thebibliography}{BCW82}

\bibitem[BCW82]{BCW}  H. Bass, E. Connell, D. Wright,
{\it The Jacobian conjecture: reduction of degree and formal expansion of the inverse,}
Bull. of the A.M.S. {\bf 7} (1982), 287-330.


\bibitem[BCM12]{BCM}
C. Bisi, A. Calabri, M. Mella, {\it On plane Cremona transformations of fixed degree}, 
http://arxiv.org/abs/1212.0996

\bibitem[Bla10]{Bl} J. Blanc, 
{\it Groupes de Cremona, connexit\'e et simplicit\'e.}
Ann. Sci. Ec. Norm. Sup\'er. (4) {\bf 43} (2010), no. 2, 357-364.

\bibitem[Bou98]{Bou}
N. Bourbaki, {\it General topology.} Chapters 1--4.
Elements of Mathematics. Springer-Verlag, Berlin, 1998. 


\bibitem[Bri10]{Brion} M. Brion, 
{\it Some basic results on actions of nonaffine algebraic groups}, Symmetry and spaces, 1--20, Progr. Math. 278, Birkh\"{a}user Boston, Inc., Boston, MA, 2010. 



\bibitem[CL12]{Cantat-Lamy} S. Cantat, S. Lamy, 
{\it Normal subgroups of the Cremona group},
Acta Mathematica (to appear).
arXiv:1007.0895v2




\bibitem[CD08]{CD} D. Cerveau, J. D\'eserti,
{\it Transformations birationnelles de petit degr\'e.}
{Cours Sp\'ecialis\'es, Soci\'et\'e Math\'ematique de France} (to appear). arXiv:0811.2325


\bibitem[Dem70]{De}  M. Demazure, 
{\it Sous-groupes alg\'ebriques de rang maximum du groupe de Cremona,}
Ann. Sci. \'Ecole Norm. Sup. (4) {\bf 3} (1970), 507-588.

\bibitem[Fav10]{Fa} C. Favre, 
{\it Le groupe de Cremona et ses sous-groupes de type fini.}
S\'eminaire Bourbaki. Volume 2008/2009.  
Ast\'erisque No. {\bf 332} (2010), Exp. No. 998, vii, 11--43.

\bibitem[Fra65]{Fra}S. P. Franklin, {\it
Spaces in which sequences suffice.}  Fund. Math. {\bf 57} (1965) 107--115. 

\bibitem[Gro67]{Gro}
A. Grothendieck, {\it \'El\'ements de g\'eom\'etrie alg\'ebrique. IV. \'Etude locale des sch\'emas et des morphismes de sch\'emas IV. } Inst. Hautes \'Etudes Sci. Publ. Math. No. {\bf 32}  1967. 


\bibitem[Hum81]{Humphreys} J. E. Humphreys, Linear algebraic groups (2nd ed.),
Graduate texts in Mathematics 21, Springer Verlag, 1981.



\bibitem[Kam96]{Kam1} T. Kambayashi, {\it  Pro-affine algebras, ind-affine groups and the Jacobian problem}, J. Algebra
{\bf 185} (1996), no. 2, 481--501.

\bibitem[Kam03]{Kam2}
T. Kambayashi,  {\it Some basic results on pro-affine algebras and ind-affine schemes}, Osaka J.
Math. {\bf 40} (2003), no. 3, 621--638.


\bibitem[Mum74]{Mum}
D. Mumford, {\it Algebraic Geometry} in {\it Mathematical developments arising from Hilbert problems.} 
Proceedings of the Symposium in Pure Mathematics of the American Mathematical Society held at Northern Illinois University, De Kalb, Ill., May, 1974. 44--45.



\bibitem[Ngu09]{Ngu}
D. Nguyen, {\it Groupe de Cremona}, PhD thesis, Universit\'e de Nice-Sophia Antipolis 2009.

\bibitem[Pan99]{Pan}
I. Pan, {\it Une remarque sur la g\'en\'eration du groupe de Cremona. } Bol. Soc. Brasil. Mat. (N.S.) {\bf 30} (1999), no. 1, 95--98. 

\bibitem[PRV01]{PRV}
I. Pan, F. Ronga, T. Vust, {\it Transformations birationnelles quadratiques de l'espace projectif complexe \`a trois dimensions.} 
Ann. Inst. Fourier {\bf 51} (2001), no. 5, 1153--1187. 

\bibitem[Ser10]{Se}
J.-P. Serre, {\it Le groupe de Cremona et ses sous-groupes finis.} 
S\'eminaire Bourbaki. Volume 2008/2009. Ast\'erisque No. {\bf 332} (2010), Exp. No. 1000, vii, 75--100.



\bibitem[Sha66]{Sha1} I. R. Shafarevich, {\it On some infinite-dimensional groups},
Rend. Mat. e Appl. (5) 25 (1966), no. {\bf 1-2}, 208-212.



\bibitem[Sha82]{Sha2} I. R. Shafarevich, {\it On some infinite-dimensional
groups II}, Math. USSR Izv. {\bf 18} (1982), 214-226.

\bibitem[Sta12]{Sta} I. Stampfli, {\it On the topologies on ind-varieties and related irreducibility questions.} J. Algebra {\bf 372} (2012), 531--541. 

\bibitem[Wey39]{Weyl}
H. Weyl, {\it On unitary metrics in projective space.} Ann. of Math. (2) {\bf 40} (1939), no. 1, 141--148.


\bibitem[Whi48]{White}
J.H.C. Whitehead, {\it Note on a theorem due to Borsuk. }
Bull. Amer. Math. Soc. {\bf 54}, (1948). 1125--1132. 







\end{thebibliography}
\end{document}